\documentclass[a4paper]{article}

\usepackage[T1]{fontenc}
\usepackage[utf8]{inputenc}
\usepackage{lmodern}
\usepackage[colorlinks]{hyperref}
\usepackage[english]{babel}

\usepackage{tikz}
\usepackage{subcaption}

\usepackage{amssymb, amsfonts, amsmath, amsthm, mathrsfs,bbm,mathtools}
\theoremstyle{plain}
\newtheorem{theorem}{Theorem}[section]
\newtheorem{corollary}[theorem]{Corollary}
\newtheorem{lemma}[theorem]{Lemma}
\newtheorem{proposition}[theorem]{Proposition}

\newtheorem{conjecture}[theorem]{Conjecture}
\newtheorem{alttheorem}{Theorem}

\theoremstyle{remark}
\newtheorem{remark}[theorem]{Remark}

\numberwithin{equation}{section}

\usepackage{graphicx}
\def\shrug{\texttt{\raisebox{0.75em}{\char`\_}\char`\\\char`\_\kern-0.5ex(\kern-0.25ex\raisebox{0.25ex}{\rotatebox{45}{\raisebox{-.75ex}"\kern-1.5ex\rotatebox{-90})}}\kern-0.5ex)\kern-0.5ex\char`\_/\raisebox{0.75em}{\char`\_}}}

\newcommand{\N}{\mathbb{N}}
\newcommand{\Z}{\mathbb{Z}}

\newcommand{\R}{\mathbb{R}}

\newcommand{\ind}[1]{\mathbbm{1}_{\left\{#1\right\}}}

\newcommand{\crochet}[1]{{\left\langle #1 \right\rangle}}
\renewcommand{\bar}[1]{\overline{#1}}
\renewcommand{\tilde}[1]{\widetilde{#1}}

\renewcommand{\phi}{\varphi}
\renewcommand{\epsilon}{\varepsilon}

\newcommand{\E}{\mathbb{E}}
\renewcommand{\P}{\mathbb{P}}

\newcommand{\dd}{\mathrm{d}}
\newcommand{\egaldistr}{\stackrel{(d)}{=}}

\renewcommand{\rho}{\varrho}

\newcommand{\x}{\mathbf{x}}

\usepackage{xcolor}

\title{On the branching convolution equation $\mathcal E = \mathcal{Z} \circledast \mathcal E$}
\author{Pascal Maillard\thanks{Institut de Mathématiques de Toulouse, CNRS, UMR5219, Université de Toulouse, 118 route de Narbonne, 31062 Toulouse cedex 09, France. Supported in part by grants ANR-20-CE92-0010-01 and ANR-11-LABX-0040 (ANR program ``Investissements d'Avenir'').} \and Bastien Mallein\thanks{LAGA, UMR 7539, Universit\'e Sorbonne Paris Nord, 99 avenue Jean-Baptiste Clément, F-93430, Villetaneuse, France. Partially supported by the ANR grant MALIN (ANR-16-CE93-0003).}}

\date{\today}

\begin{document}

\maketitle

\begin{abstract}
We characterize all random point measures which are in a certain sense stable under the action of branching. Denoting by $\circledast$ the \emph{branching convolution operation} introduced in \cite{BM19}, and by $\mathcal{Z}$ the law of a random point measure on the real line, we are interested in solutions to the fixed point equation
\[
\mathcal E = \mathcal{Z} \circledast \mathcal E,
\]
with $\mathcal E$ a random point measure distribution. Under suitable assumptions, we characterize all solutions of this equation as shifted decorated Poisson point processes with a uniquely defined shift.
\end{abstract}



\section{Introduction}

In the recent years, the asymptotic behaviour of extremal particles in branching processes has been the subject of a large literature. Starting from the results of Bramson \cite{Bra78} and Lalley and Sellke \cite{LaS87} on the asymptotic behaviour of the maximum of branching Brownian motions, the convergence in distribution of the position of the largest atom in branching random walks was obtained by Aïdékon \cite{Aid13}. In both cases, the limiting distribution was identified as a randomly shifted Gumbel distribution.

More generally, the convergence in distribution of the extremal process and the study of the limit has been a recurrent subject of interest in the last decade. This convergence was obtained for branching Brownian motion by Aïdékon et al. \cite{ABBS13} and Arguin et al. \cite{ABK13}. Madaule \cite{Mad17}, relying on \cite{Mai13}, extended this convergence result to branching random walks: the extremal process of the branching random walk converges to a randomly shifted Poisson point process with exponential intensity, decorated by i.i.d. copies of a random point measure. This limiting point process was referred to in \cite{SuZ15} as a \emph{randomly Shifted Decorated Poisson Point Process}, or SDPPP. 
  
The convergence of extremal processes to SDPPP has since been observed in a variety of processes presenting a branching-type structure. Among these results, we can mention
\begin{itemize}
  \item time-inhomogeneous branching Brownian motions \cite{BoH15,MaZ16}, in which the variance of particles at time $s$ depends continuously on $s/t$,
  \item generalized Random Energy Model \cite{ScK15}, in which particles branch for $n^\alpha$ steps before displacing independently for $n^\alpha$ steps,
  \item discrete Gaussian free field on the square \cite{BiL16}, which is a Gaussian field on $\Z^2 \cap [0,N]$ with correlation given by the Green kernel of the random walk killed when hitting the boundary,
  \item catalytic branching Brownian motion \cite{Boch19} in which particles move as i.i.d. Brownian motions but can only branch in a neighbourhood of the origin,
  \item randomly decorated branching random walks \cite{BaP}, in which the extremal process of a branching random walk is decorated by i.i.d. random variables with exponential tail,
  \item reducible multitype branching Brownian motion \cite{BM20}, in which two types of particles are present, particles of type $1$ giving birth to particles of type $2$ but not reciprocally.
\end{itemize}
In all cases the limit of the suitably translated extremal process has been proved to converge to an SDPPP. The aim of this article is to present a unifying theory. We consider random point measures which are in a certain sense stable under the action of branching. We characterize these processes as SDPPP with an explicit random shift. As argued below, such a characterization should allow to streamline proofs of convergence of extremal processes of branching-type systems.

\subsection{Point measures and branching random walks}

We denote by $\mathfrak{P}$ the set of point measures on $\R$ assigning finite mass to the interval $(0,\infty)$. Specifically, $D \in \mathfrak{P}$ if and only if $D$ is a measure on $\R$ with an integer-valued tail, i.e.
\[
  \bar{D}(x) \coloneqq D((x,\infty)) \in \Z_+ \quad \text{ for every } x \in \R.
\]
Given $D \in \mathfrak{P}$ we call the \emph{ranked sequence of atoms} of $D$ the non-increasing sequence $\mathbf{d} = (d_n, n \in \N)$ of the positions of the atoms of $D$, repeated according to their multiplicity. In order to account for finite point measures, we write $d_n = - \infty$ if $D(\R) < n$. In the rest of this article, we canonically identify the point measure $D$ with its ranked sequence of atoms $\mathbf{d}$ via
\[
  D = \sum_{n =1}^\infty \ind{d_n > -\infty}\delta_{d_n} \text{ $\iff$ } \mathbf{d} = \left(\sup\{y \in \R : D((y,\infty)) < n\}, n \in \N \right).
\]
Therefore we canonically identify $\mathfrak{P}$ with the space of non-increasing sequences in $[-\infty,\infty)$ such that $\lim_{n \to \infty} d_n = -\infty$, with the identification $\mathbf{d} = D$. In particular, the null measure is identified with the sequence $(-\infty,-\infty,\cdots)$, and a Dirac mass at point $a \in \R$ with $(a,-\infty,-\infty,...)$.

Given $D \in \mathfrak{P}$ and $f$ a measurable non-negative function, we set
\[
  \crochet{D,f} = \int f \dd D = \sum_{n=1}^\infty \ind{d_n > -\infty} f(d_n).
\]
We also denote by $\tau$ the translation operator on $\mathfrak{P}$, defined for all $y \in [-\infty, \infty)$ via
\[
  \tau_y D = \tau_y \mathbf{d} = (d_n + y, n \in \N).
\]

The law of a random point measure $E$ (i.e. a random element of $\mathfrak{P}$) is characterized by its Laplace functional (see e.g.~\cite[Section~9.4]{DV2008}), defined as $\phi \in \mathfrak{T} \mapsto \E\left( \exp\left( - \crochet{E,\phi} \right) \right)$, where $\mathfrak{T}$ is a large enough set of test functions. In this article, we denote by $\mathfrak{T}$ the set of continuous non-negative bounded functions on $\R$ with support bounded on the left.

A \emph{branching random walk} is a $\mathfrak{P}$-valued stochastic process $(Z_n, n \geq 0)$, starting from $Z_0 = \delta_0$, that satisfies the branching property, i.e. such that for all $n,m$
\begin{equation}
  \label{eqn:branchingProp}
  Z_{n+m} \stackrel{(d)}= \sum_{k=1}^\infty \tau_{z_k} Z^{(k)}_m, \text{ where $(z_1,z_2\ldots)$ is the ranked sequence of atoms of $Z_n$},
\end{equation}
and $(Z^{(k)}_m, k \in \N)$ are i.i.d. copies of $Z_m$, further independent of $Z_n$. The law of the branching random walk can then be described by recurrence using only the law of $Z\coloneqq Z_1$ on $\mathfrak{P}$, that we write $\mathcal{Z}$. A branching random walk can be described as a particle system on $\R$ starting from a single particle at the origin at time $0$. At each generation, every particle creates offspring around its position, according to an independent point measure of law $\mathcal{Z}$, translated by its position. 

Note that in this article, we do not study branching random walks \emph{per se}, rather, they appear as a tool in the study of point measure distributions which are in a sense stable under the action of branching.

To ensure the well-definition of the branching random walk $(Z_n)_{n\ge0}$ (i.e. that $Z_n \in \mathfrak{P}$ for all $n \in \N$ a.s.), one usually further requests the non-degeneracy of the Laplace transform of the intensity measure of $Z$. More precisely, we assume the existence of $\theta > 0$ such that
\begin{equation}
  \label{eqn:integrabilityCondition}
  \kappa(\theta) \coloneqq \log \int \crochet{Z,\exp_\theta} \mathcal{Z}(\dd Z) = \log \E\left( \sum_{j=1}^\infty e^{\theta z_j}\right) < \infty,
\end{equation}
writing $\exp_\theta : z \mapsto e^{\theta z}$. 
Under this assumption and the branching property, one straightforwardly obtains that for all $n \in \N$, $\displaystyle \E\left( \crochet{Z_n,\exp_\theta} \right) = e^{n \kappa(\theta)}$, and
\begin{equation}
  \label{eqn:defineMartingale}
  \left( \crochet{Z_n,\exp_\theta} e^{-n \kappa(\theta)}, n \geq 0 \right) \text{ is a non-negative martingale.}
\end{equation}

The branching property \eqref{eqn:branchingProp} can be written more concisely using the notion of  \emph{branching convolution operation} $\circledast$, which acts on probability distributions on~$\mathfrak{P}$. The branching convolution operation was introduced in \cite{BM19} and some of its properties studied there. This operation is an extension to probability distributions on $\mathfrak{P}$ of the usual convolution operation for measures on $\R$. Given $\mathcal{D}$ and $\mathcal{E}$ two probability distributions on $\mathfrak{P}$, their convolution $\mathcal{D} \circledast \mathcal{E}$ is defined as the law of the point measure
\[
  \sum_{j = 1}^\infty \tau_{d_j} E^{(j)},
\]
where $(d_j, j \geq 1)$ is the ranked sequence of atoms of a point measure of law $\mathcal{D}$ and $(E^{(j)})$ are i.i.d. random point measures with law $\mathcal{E}$.  It is easy to check that the branching convolution operation is associative, and that the law of the Dirac measure $\delta_0$ (i.e.~the measure $\delta_{\delta_0}$) is an identity. Denoting by $\mathcal{Z}_n$ the law of $Z_n$ and by $\mathcal{Z}$ the law of $Z = Z_1$, we then have
\[
\mathcal{Z}_n = \mathcal{Z}^{\circledast n} = \mathcal{Z}_{n-1} \circledast \mathcal{Z} = \mathcal{Z} \circledast \mathcal{Z}_{n-1}.
\]

\paragraph{Notation conventions.}
In this article, we use as much as possible the following typographic convention when dealing with random point measures:
\begin{itemize}
  \item calligraphic capital letters $\mathcal{D},\mathcal{Z},\ldots$ are reserved for laws of a random point measures, i.e. probability distributions on $\mathfrak{P}$;
  \item roman upper case letters $D,Z,\ldots$ are used to represent random point measures with corresponding distribution, i.e. random variables in the space $\mathfrak{P}$ with law $\mathcal{D}$;
  \item roman lower case $d_1, z_2, \ldots$ is used to denote the positions of the largest, second largest, etc. atoms of the point measure $D$;
  \item boldface lower case $\mathbf{d} = (d_n, n \in \N),\mathbf{z},\ldots$ is used for the full ranked sequence of atoms.
\end{itemize}

\subsection{Shifted decorated Poisson point processes}

Given $S$ a non-negative random variable, $\alpha > 0$ and $\mathcal{D}$ a probability distribution on $\mathfrak{P}$, an SDPPP with parameters ($S$, $e^{-\alpha x} \dd x$, $\mathcal{D}$) is constructed as follows. Let $(\xi_j, j \in \N)$ denote the ranked sequence of atoms of a Poisson point process with intensity $e^{-\alpha x} \dd x$ and let $(D_i, i \in \N)$ be i.i.d. random point measures with law $\mathcal{D}$. The SDPPP($S$, $e^{-\alpha x} \dd x$, $\mathcal{D}$) is the random point measure defined as
\begin{equation}
  \label{eqn:SDPPPdef}
  E = \sum_{i =1}^\infty \tau_{\alpha^{-1}\log S + \xi_i} D_i,
\end{equation}
provided that $E \in \mathfrak{P}$ almost surely. Using the branching convolution operation, the law $\mathcal{E}$ of the SDPPP defined in \eqref{eqn:SDPPPdef} can be written as
\begin{equation}
  \label{eqn:branchingDecomp}
\mathcal E = \mathcal S \circledast \mathcal X \circledast \mathcal D,
\end{equation}
where $\mathcal S$ is the law of the random point measure $\delta_{\alpha^{-1}\log S}$ and $\mathcal X$ is the law of the Poisson process $(\xi_j, j \in \N)$. Remark that convolving on the left with the law of a (random) Dirac measure precisely amounts to a shift by the position of the atom of this Dirac measure.
 A systematic study of SDPPP was undergone in \cite{SuZ15}, we recall some of the properties of these random point measures below.

The Laplace functional of an SDPPP $E$ with parameters ($S$, $e^{-\alpha x} \dd x$, ${\mathcal{D}}$) can be written
\begin{equation}
  \label{eqn:LTSDPPP}
  \E\left( e^{-\crochet{E,\phi}} \right) = \E\left( \exp\left( -S \int_\R e^{-\alpha x} (1 - e^{-\Psi_\phi(x)}) \dd x\right) \right),
\end{equation}
where $\displaystyle \Psi_\phi(x) = - \log \E\left( e^{-\crochet{\tau_x D,\phi}}\right)$ and $D$ is a random point measure with law $\mathcal{D}$. Subag and Zeitouni proved in particular that a characteristic property of SDPPP is the existence of a non-increasing function $g$ such that for every continuous non-negative compactly supported function $\phi$, there exists $T_\phi \in \R$ such that
\begin{equation}
  \label{eqn:pointMeasure}
  \forall y \in \R, \quad \E\left( e^{-\crochet{\tau_y E,\phi}} \right) = \E\left( e^{-\crochet{E,\phi(\cdot + y)}} \right) = g(y - T_\phi).
\end{equation}
This result is stated more precisely in Theorem~\ref{thm:SuZ} below. The reader can easily check that for the SDPPP given above, the function $g(y) = \E[\exp(-Se^{\alpha y})]$ satisfies these conditions.

Formula \eqref{eqn:SDPPPdef} is well-defined if and only if the law of $d_1 \eqqcolon \max D$, the position of the largest atom in $D$, has an exponential moment of order $\alpha$, i.e.
\begin{equation}
  \label{eqn:conditionForIntegrability}
  c \coloneqq \E(e^{\alpha d_1}) = \int_{\mathfrak{P}} e^{\alpha \max D} \mathcal{D}(\dd D) < \infty.
\end{equation}
In this case, we define the law $\mathcal{D}^\ast$ on $\mathfrak P$ by
\begin{equation}
  \label{eqn:changeOfMeasure}
  \int_{\mathfrak{P}} f(D) \mathcal{D}^\ast(\dd D) = \frac{\int_{\mathfrak{P}} f(\tau_{-\max D} D) e^{\alpha \max D} \mathcal{D}(\dd D)}{\int_{\mathfrak{P}} e^{\alpha \max D} \mathcal{D}(\dd D)}
\end{equation}
for all measurable bounded functions $f$. Remark that the law $\mathcal{D}^\ast$ is supported by $\displaystyle \mathfrak{P}^\ast \coloneqq \left\{ \x \in \mathfrak{P} : x_1 = 0\right\}$, and $E$ is also an SDPPP($c S$, $e^{-\alpha x} \dd x$, $\mathcal{D}^\ast$), as one readily checks. In other words, we can assume without loss of generality that the law of the decoration of an SDPPP is supported on $\mathfrak{P}^\ast$.

Under some integrability conditions on $S$ (namely \eqref{eqn:regularityLaplace} below), one can recover the law $\mathcal{D}^\ast$ from the law of $E$ via the following formula
\begin{equation}
  \label{eqn:obtainingTheDecoration}
  \int_\mathfrak{P} f(\x) \mathcal{D}^\ast (\dd \x )= \lim_{z \to \infty} \E\left( f\left( \tau_{-\max E} E \right) \middle| \max E \geq z \right),
\end{equation}
where $\max E$ is the position of the largest atom in $E$. Additionally,
\begin{equation}
  \label{eqn:lawMax}
  \P(\max E \leq z) = \P\left( cS + \max_{j \in \N} \xi_j \leq x \right) = \E\left( \exp\left( - c S e^{-\alpha z} \right) \right),
\end{equation}
using that $\xi_1 = \max_{j \in \N} \xi_j$ follows the Gumbel distribution. As a result, given the value of $\alpha$, both the law of $cS$ and the law of $\mathcal{D}^\ast$ can be identified from the law of $E$.

\subsection{Fixed points of a branching convolution equation}
\label{subsec:13}

In this section, we introduce branching convolution equations and state the main result of this article.
Writing $(z_n, n \geq 1)$ for the ranked sequence of atoms of $Z$, we assume there exists $\alpha > 0$ such that
\begin{equation}
  \tag{A1}
  \label{eqn:regularityAssumption}
  \E\left( \sum_{j =1}^\infty \ind{z_j >-\infty} \right) > 1 \quad \text{and} \quad \E\left( \sum_{j=1}^\infty e^{\alpha z_j} \right) = 1.
\end{equation}
The first condition ensures that $\P(\forall n \in \N, Z_n \neq 0) > 0$, i.e. the survival of the branching random walk with positive probability, and the second one that the process $(\crochet{Z_n,\exp_\alpha}, n \geq 1)$ is a non-negative martingale, that we refer to as the additive martingale. Next, we assume one of the two following conditions holds
\begin{equation}
  \tag{A2a}
  \label{eqn:regularCase}
  \E\left( \sum_{j=1}^\infty z_j e^{\alpha z_j}\right) \in (-\infty,0) \quad \text{and} \quad \E\left( \crochet{Z,\exp_\alpha} \log_+ \crochet{Z,\exp_\alpha} \right) < \infty,
\end{equation}
\begin{multline}
  \tag{A2b}
  \label{eqn:boundaryCase}
  \qquad \qquad \E\left( \sum_{j=1}^\infty z_j e^{\alpha z_j}\right) = 0, \quad \E\left(  \sum_{j=1}^\infty z_j^2 e^{\alpha z_j} \right) < \infty \\ \text{and} \quad \E\left( \crochet{Z,\exp_\alpha} \log_+ \crochet{Z,\exp_\alpha}^2 \right) + \E(\tilde{X}\log_+\tilde{X}) < \infty,\qquad \qquad 
\end{multline}
where $\tilde{X} = \sum_{j=1}^\infty (z_j)_+ e^{\alpha z_j}$. If assumption \eqref{eqn:regularCase} holds, we say that we are in the regular case, in opposition with the boundary case when \eqref{eqn:boundaryCase} holds, following the terminology of \cite{BiK05}.

Under assumption \eqref{eqn:regularCase}, the martingale $\crochet{Z_n,\exp_\alpha}$ is uniformly integrable and converges a.s. to a non-degenerate limit that we write $S$ (see Biggins \cite{Big77}, Lyons \cite{Lyo97}), whereas under assumption \eqref{eqn:boundaryCase}, it converges to $0$ almost surely. In that situation, writing $f : x \mapsto x e^{\alpha x}$, the process $(\crochet{Z_n,f})$ is a (signed, non-uniformly integrable) martingale converging almost surely to a non-degenerate, non-negative limit that we also write $S$ (see \cite{Aid13}, \cite{Che15}). This martingale is called the derivative martingale. We observe that in both cases, the variable $S$ satisfies
\begin{equation}
  \label{eqn:fixedPointEquation}
  S \egaldistr \sum_{j=1}^\infty e^{\alpha z_j} S^{(j)},
\end{equation}
where $(S^{(j)}, j \in \N)$ are i.i.d. copies of $S$ that are further independent of $Z$.

In the rest of the article, we assume
\begin{equation}
  \tag{A2}
  \label{eqn:bothCase}
  \text{\eqref{eqn:regularCase} or \eqref{eqn:boundaryCase} holds,}
\end{equation}
and we write $S$ the limit of the additive or the derivative martingale depending on whether we are in the regular or the boundary case. Finally, we assume the branching random walk to be non-lattice, i.e.
\begin{equation}
  \tag{A3}
  \label{eqn:nonLattice}
  \forall a > 0, \ \forall b \in \R, \ \P(\forall j \in \N, z_j \in a \Z+ b) < 1.
\end{equation}

Under assumptions \eqref{eqn:regularityAssumption}, \eqref{eqn:bothCase} and \eqref{eqn:nonLattice} we take interest in random point measures $E$ satisfying the following equality in distribution
\begin{equation}
  \label{eqn:stablePointMeasure}
  E \stackrel{(d)}= \sum_{j =1}^\infty \tau_{z_j} E^{(j)},
\end{equation}
where $(E^{(j)}, j \in \N)$ are i.i.d. copies of $E$ further independent of $\mathbf{z}= Z$.

Then, denoting by $\mathcal{Z}$ the law of the first generation of the branching random walk, we can reformulate the problem of identifying the solutions $E$ to the equation in law \eqref{eqn:stablePointMeasure} as the identification of probability measures $\mathcal{E}$ on $\mathfrak{P}$ satisfying
\begin{equation}
  \label{eqn:stableLaws}
  \mathcal{E} = \mathcal{Z} \circledast \mathcal{E}.
\end{equation}

The main result of the article is the following characterization of the solutions of this fixed point equation.
\begin{theorem}
\label{thm:main}
Under assumptions \eqref{eqn:regularityAssumption}, \eqref{eqn:bothCase} and \eqref{eqn:nonLattice}, a random point measure $E$ satisfies \eqref{eqn:stablePointMeasure} if and only if there exist $c>0$ and a probability distribution $\mathcal{D}^\ast$ on $\mathfrak{P}^\ast$ such that $E$ is a SDPPP($cS$, $e^{-\alpha x}\dd x$, $\mathcal{D}^\ast$). 
\end{theorem}

Equivalently, Theorem~\ref{thm:main} can be reformulated as follows: under assumptions \eqref{eqn:regularityAssumption}, \eqref{eqn:bothCase} and \eqref{eqn:nonLattice} a distribution $\mathcal{E}$ on $\mathfrak{P}$ satisfies \eqref{eqn:stablePointMeasure} if and only if there exists a positive constant $c$ and a probability distribution $\mathcal{D}^*$ on $\mathfrak{P}^\ast$, such that $\mathcal{E} = \mathcal{S}_c \circledast \mathcal{X} \circledast \mathcal{D}^*$, where $\mathcal{S}_c$ is the law of the random point measure $\delta_{\alpha^{-1}\log(cS)}$ and $\mathcal{X}$ is the law of a Poisson process with intensity $e^{-\alpha x}\dd x$.

\begin{remark}
\label{rem:branchingAfter}
Note that if $\mathcal{E}$ is a point measure distribution satisfying \eqref{eqn:stableLaws}, then for any probability distribution $\mathcal{F}$ on $\mathfrak{P}^\ast$, the point measure defined as $\bar{\mathcal{E}} = \mathcal{E} \circledast \mathcal{F}$ also satisfies \eqref{eqn:stableLaws}. As a result, satisfying \eqref{eqn:stableLaws} alone does not characterize $\mathcal{D}^\ast$.
\end{remark}

Theorem \ref{thm:main} can be used to characterize the law of the limit of the extremal point measure of a branching random walk, provided that its convergence was already proved. Using the branching property \eqref{eqn:branchingProp} at the first generation, i.e. with $m \in \N$ and $n = 1$, then letting $m \to \infty$, we deduce that the limiting point measure has to satisfy \eqref{eqn:stablePointMeasure}. For example, if the extremal process of the time-inhomogeneous branching Brownian motion studied in \cite{MaZ16} were to converge, the limit would have to be an SDPPP, with the normalization constant $c>0$ and the decoration law $\mathcal{D}$ to be determined.

We conjecture that applying the branching property at its last step, i.e. for $m=1$ and $n \in \N$ would allow to complete the characterization of the SDPPP by identifying the limiting point measure. Hence, we are interested in solutions of the following dual fixed point equation:
\begin{equation}
\label{eqn:other_equation}
  \mathcal{E} = \mathcal{E} \circledast \mathcal{Z},
\end{equation}
which can be written as
\[
E \egaldistr \sum_{i = 1}^\infty \sum_{j=1}^\infty \delta_{e_i + z^i_j},
\]
where $\mathbf{e}$ is the ranked sequence of atoms of $E$, a point measure of law $\mathcal{E}$, and $(\mathbf{z}^i, i \in \N)$ are the ranked sequences of atoms of i.i.d. copies of $Z$.

\begin{conjecture}
\label{conj:WouldBeCool}
Assume \eqref{eqn:regularityAssumption} and \eqref{eqn:nonLattice} and that either \eqref{eqn:boundaryCase} or
\begin{equation}
  \tag{A2c}
  \label{eqn:superCriticalCase}
  \E\left( \sum_{j=1}^\infty z_j e^{\alpha z_j}\right) \in (0,\infty) \quad \text{and} \quad \E\left( \crochet{Z,\exp_\alpha} \log_+ \crochet{Z,\exp_\alpha} \right) < \infty.
\end{equation}
There exists a law $\mathcal{D}^\ast$ on $\mathfrak{P}^\ast$, such that the following holds: a probability distribution $\mathcal{E}$ on $\mathfrak{P}$ satisfies \eqref{eqn:other_equation} if and only if there exists a non-negative random variable $\tilde{S}$ such that $\mathcal{E}$ is the law of an SDPPP($\tilde{S}$, $e^{-\alpha x}\dd x$, $\mathcal{D}^\ast$).
\end{conjecture}

Equivalently, this conjecture states that any point measure distribution $\mathcal{E}$ satisfying $\mathcal{E} = \mathcal{E} \circledast \mathcal{Z}$ can be factorized as $\mathcal{E} = \tilde{\mathcal{S}}\circledast\mathcal{X} \circledast \mathcal{D}^\ast$, with $\mathcal{D}^\ast$ uniquely defined, $\mathcal{X}$ the law of a Poisson process with intensity $e^{-\alpha x}\dd x$, and $\tilde{\mathcal S}$ the law of some random shift. If the conjecture holds, then combined with Theorem~\ref{thm:main}, a point measure distribution such that $\mathcal{E} = \mathcal{E} \circledast \mathcal{Z}_1 = \mathcal{Z}_2 \circledast \mathcal{E}$ will be uniquely defined, up to translation by a constant.

The conjecture is consistent with the convergence results obtained by Bovier and Hartung \cite{BoH15} and Maillard and Zeitouni \cite{MaZ16} for time-inhomogeneous branching Brownian motions. In \cite{BoH15}, the limiting point measure satisfies $\mathcal{E} = \mathcal{E} \circledast \mathcal{Z}_1 = \mathcal{Z}_2 \circledast \mathcal{E}$ with $\mathcal{Z}_2$ verifying \eqref{eqn:regularCase} and $\mathcal{Z}_1$ verifying \eqref{eqn:superCriticalCase}. In \cite{MaZ16} the candidate limit for the extremal process would satisfy similar equalities with both $\mathcal{Z}_1$ and $\mathcal{Z}_2$ satisfying \eqref{eqn:boundaryCase}.

Conjecture \ref{conj:WouldBeCool} has recently been proved by \cite{CGS20} for branching Brownian motions with drift, under the analogue of assumption \eqref{eqn:boundaryCase}. They discuss in this article the possible adaptation of the same techniques to branching random walk settings. Kabluchko \cite{Kab} obtained convergence results for branching random walks satisfying assumption \eqref{eqn:superCriticalCase} starting from an initial condition given by a Poisson process with exponential intensity.

\section{Proof of Theorem \ref{thm:main}}

In the rest of the article, $S$ will stand from the limit of the additive or the derivative martingale, depending on whether \eqref{eqn:regularCase} or \eqref{eqn:boundaryCase} hold, as in Section~\ref{subsec:13}. The proof of Theorem \ref{thm:main} is decomposed into two largely independent parts: we first show that for all $c > 0$ and $\mathcal{D}$ a probability distribution on $\mathfrak{P}^*$, the SDPPP($cS$, $e^{-\alpha x}\dd x$, $\mathcal{D}$) is a solution to \eqref{eqn:stablePointMeasure}, by straightforward computation of its Laplace transform. Then, we show that any solution to \eqref{eqn:stablePointMeasure} can be decomposed as a SDPPP, using their characterization in \cite{SuZ15}.

\subsection{Theorem \ref{thm:main}: Sufficiency}

As a first step towards the proof of Theorem~\ref{thm:main}, we observe that a randomly shifted Poisson point process is a solution to \eqref{eqn:stablePointMeasure}.
\begin{lemma}
\label{lem:existence}
Under assumptions \eqref{eqn:regularityAssumption}, \eqref{eqn:bothCase} and \eqref{eqn:nonLattice}, for all $c > 0$ the law of the Cox process $E$ with intensity $c S e^{-\alpha x} \dd x$ verifies \eqref{eqn:stablePointMeasure}.
\end{lemma}

\begin{proof}
This result is a simple consequence of the superposability property of Poisson point process. Given $(E^{(i)}, i \in \N)$ i.i.d. copies of $E$ with intensity $S^{(i)}e^{-\alpha x} \dd x$, we observe that $\sum_{i=1}^\infty \tau_{z_i} E^{(i)}$ is a Cox process with intensity $\sum_{i=1}^\infty c S^{(i)} e^{-\alpha (x - z_i)} \dd x = c e^{-\alpha x}\left( \sum S^{(i)} e^{\alpha z_i} \right) \dd x$. Then as $S$ satisfies \eqref{eqn:fixedPointEquation}, the proof is complete.
\end{proof}

As a corollary, we obtain that any SDPPP($c S$, $e^{-\alpha x}\dd x$, $\mathcal{D}$) satisfies \eqref{eqn:stablePointMeasure}.
\begin{corollary}
\label{cor:existence}
Let $\mathcal{D}$ be a random point measure distribution on $\mathfrak{P}^\ast$ and $c > 0$. Under assumptions \eqref{eqn:regularityAssumption}, \eqref{eqn:bothCase} and \eqref{eqn:nonLattice}, an SDPPP($c S$, $e^{-\alpha x}\dd x$, $\mathcal{D}$) satisfies \eqref{eqn:stablePointMeasure}.
\end{corollary}

\begin{proof}
It is a direct consequence of Remark \ref{rem:branchingAfter}. Denote by $\mathcal{P}$ the law of a Cox process with intensity $c S e^{-\alpha x} \dd x$ and ${\mathcal{E}}$ the law of an SDPPP($c S$, $e^{-\alpha x}\dd x$, $\mathcal{D}$), we have ${\mathcal{E}} = \mathcal{P} \circledast \mathcal{D}$ by \eqref{eqn:branchingDecomp}. Lemma~\ref{lem:existence} shows that $\mathcal{P} = \mathcal{Z} \circledast \mathcal{P}$. Hence, we deduce that $\mathcal{Z} \circledast {\mathcal{E}} = \mathcal{Z} \circledast \mathcal{P} \circledast \mathcal{D} = \mathcal{P} \circledast \mathcal{D} = {\mathcal{E}}$ by associativity of $\circledast$.
\end{proof}

\subsection{Theorem~\ref{thm:main}: Necessity}

We consider in this section a random point measure $E$ satisfying \eqref{eqn:stablePointMeasure}. For any test function $\phi \in \mathfrak{T}$, we introduce the function
\[
  F_\phi : x \in \R \mapsto \E\left( e^{-\crochet{\tau_{x} E, \phi}} \right).
\]
We recall here the following consequence of \cite[Theorem 9(a)]{SuZ15}.
\begin{alttheorem}
\label{thm:SuZ}
Assume that for every compactly supported measure $\phi \in \mathfrak{T}$, there exists $T_\phi \in \R$ such that $\displaystyle F_\phi(x) = g(x - T_\phi)$, where 
\[
  g : x \in \R \mapsto \E\left( \exp\left( - W e^{\beta x} \right) \right),
\]
for some $\beta > 0$ and $W$ a non-negative random variable. Assume that
\begin{equation}
  \label{eqn:21}
  \forall y \in \R, \lim_{x \to -\infty} \frac{1 - g(x+y)}{1-g(x)} = e^{\beta y}.
\end{equation}
Then, there exists $c>0$ such that $E$ is an $SDPPP(c W, e^{-\beta x}\dd x, \mathcal{D}^\ast)$, where $\mathcal{D}^\ast$ is the limit law constructed using the formula \eqref{eqn:obtainingTheDecoration}.
\end{alttheorem}

Observe that \eqref{eqn:21} can be rephrased in terms of the law of $W$, it corresponds to the condition
\begin{equation}
  \label{eqn:regularityLaplace}
  \lambda \mapsto  -\log \E\left( e^{-\lambda W} \right) \text{ is regularly varying at $0$ with index $1$.}
\end{equation}

The rest of the section therefore consists in proving that we are under the conditions of application of Theorem~\ref{thm:SuZ} with $W = S$ and $\beta = \alpha$. We link this problem to the characterization of the fixed point solutions to the smoothing transform. Precisely, we say that a measurable $[0,1]$-valued function $f$ is a fixed point of the smoothing transform associated to the point measure $\mathcal{Z}$ if
\begin{equation}
  \label{eqn:fixedPoint}
  \forall x > 0, f(x) = \E\left( \prod_{j \in \N} f(x e^{z_j}) \right).
\end{equation}
This formalism might be better understood observing that if $f$ is the characteristic function of a random variable $Y$, then \eqref{eqn:fixedPoint} can be rewritten
\[
  Y \egaldistr \sum_{j =1}^\infty e^{z_j} Y_j,
\]
where $(Y_j,j \in \N)$ are i.i.d. copies of $Y$.

The study of fixed points of the smoothing transform under increasingly general conditions has been the subject of a large corpus of studies, we refer to \cite{ABM12} and the references therein for an overview of the literature. We will use the characterization of fixed points of the smoothing transform obtained in that article, which can be rephrased as follows.
\begin{alttheorem}
\label{thm:ABM}
Under assumptions \eqref{eqn:regularityAssumption}, \eqref{eqn:bothCase} and \eqref{eqn:nonLattice}, let $f$ be a fixed point of the smoothing transform, satisfying \eqref{eqn:fixedPoint}. Suppose furthermore that $f$ satisfies one of the two following conditions:
\begin{enumerate}
  \item $f$ is left-continuous and non-increasing, or
  \item $f(0)=1$, $f(t) \to 1$ as $t \to 0$, and $\log (1 - f(e^{-z}))$ is uniformly continuous on $[K,\infty)$ for some $K>0$.
\end{enumerate}
Then there exists $h>0$ such that $f(t) = \E(e^{-h S t^\alpha})$ for all $t > 0$.
\end{alttheorem}

With these two results in our hands, we can now prove the necessity part of Theorem~\ref{thm:main}. We begin with the following lemma.
\begin{lemma}
\label{lem:increasingFunctions}
Let $\phi\in\mathfrak{T}$ be non-decreasing. Under assumptions \eqref{eqn:regularityAssumption}, \eqref{eqn:bothCase} and \eqref{eqn:nonLattice} there exists $T_\phi \in \R$ such that 
$\displaystyle F_\phi(x) = \E\left( \exp\left( - S e^{\alpha (x - T_\phi)} \right) \right)$.
\end{lemma}

\begin{proof}
For all $t \geq 0$, we set $f(t) = F_\phi(\log t)$. Using \eqref{eqn:stablePointMeasure}, we have
\begin{align*}
  \E\left(\prod_{j \geq 1} f(te^{z_j})\right) &= \E\left( \exp\left( -\sum_{j \geq 1} \crochet{\tau_{z_j + \log t} E^{(j)},\phi} \right) \right)\\
  &= \E\left( \exp\left( - \crochet{\tau_{\log t}\sum_{j \geq 1} \tau_{z_j} E^{(j)},\phi} \right) \right)\\
  &= \E\left( \exp\left( -\crochet{\tau_{\log t} E,\phi} \right) \right) = f(t).
\end{align*}
Additionally, as $\phi$ is continuous and non-decreasing, we obtain that $f$ is non-increasing, and continuous by the dominated convergence theorem.

Hence, by Theorem~\ref{thm:ABM}, there exists $h> 0$ such that
\[
  f(t) = \E(e^{-h S t^\alpha}) = \E\left( \exp\left( - S e^{\alpha (\log t -T_\phi)} \right) \right),
\]
setting $T_\phi = (\log h)/\alpha$. This finishes the proof.
\end{proof}

For all $x \in \R$, we set
\begin{equation}
  \label{eqn:formulaG}
  g(x) = \E\left( \exp\left( - S e^{\alpha x} \right) \right).
\end{equation}
We note that $g$ is a continuous and decreasing function. Lemma \ref{lem:increasingFunctions} shows that for any non-decreasing function $\phi \in \mathfrak{T}$, there exists $T_\phi \in \R$ such that
\[
  \E\left( \exp\left( - \crochet{\tau_z E,\phi} \right) \right) = g(t - T_\phi).
\]
We recall the asymptotic behaviour of $g$ as $t \to 0$.
\begin{alttheorem}
\label{thm:BIM}
Under assumptions \eqref{eqn:regularityAssumption}, \eqref{eqn:bothCase} and \eqref{eqn:nonLattice}, we have
\[
  g(z) =
  \begin{cases}
    1 - e^{\alpha z} (1 + o(1)) &\text{if  \eqref{eqn:regularCase} holds}\\
    1 + \alpha z e^{\alpha z} (1 + o(1)) & \text{ if \eqref{eqn:boundaryCase} holds}
  \end{cases}
  \quad
  \text{ as $z \to -\infty$.}
\]
\end{alttheorem}
\begin{proof}
The asymptotic behaviour of $g(z)$ as $z \to -\infty$ under assumption \eqref{eqn:regularCase} is a direct consequence of the fact that $\E(S) = 1$ as the martingale $(\crochet{Z_n,\exp_\alpha})$ is uniformly integrable, see Section~\ref{subsec:13}. Under assumption \eqref{eqn:boundaryCase}, we use \cite[Theorem 2.1]{BIM} to compute an equivalent for $1-g(z)$ as $z \to -\infty$.
\end{proof}

Using Lemma \ref{lem:increasingFunctions} and Theorem~\ref{thm:BIM}, we can obtain estimates on the law of $E$. In particular, we obtain the law of $\max E$.
\begin{proposition}
\label{prop:maxDis}
There exists $c_m > 0$ such that for all $x \in \R$
\[
  \P(\max E \leq x) = \E\left(  e^{- c_m S e^{-\alpha x}}\right).
\]
In particular, $z \mapsto -\log \P(\max E > z)$ is uniformly continuous on $[K, \infty)$ for some $K$ large enough.
\end{proposition}

\begin{proof}
For $\lambda > 0$, we set
\[
  \phi_\lambda(x) = \max(0, \min(1,\lambda x)).
\]
Applying the monotone convergence theorem, for all $z \in \R$ and $\mu > 0$, we have
\[
  \lim_{\lambda \to \infty} F_{\mu \phi_\lambda}(z) = \E\left( e^{- \mu E((-z,\infty))} \right).
\]
By Lemma \ref{lem:increasingFunctions}, for all $\lambda,\mu \geq 0$, there exists $T_{\lambda,\mu}$ such that $F_{\mu \phi_\lambda}(z) = g(z - T_{\lambda,\mu})$. More precisely, as $g$ is continuous, decreasing and $\lim_{z \to -\infty} g(z) = 1$, we have $T_{\lambda,\mu} = - g^{-1}(F_{\mu \phi_\lambda}(0))$. As $g^{-1}$ is continuous, $T_{\lambda,\mu}$ converges, as $\lambda \to \infty$ to
\[
  T_{\infty,\mu} = g^{-1}\left( \E\left( e^{- \mu E((-z,\infty))} \right)\right).
\]
As a result, we deduce that for all $\mu > 0$, we have
\begin{equation}
  \label{eqn:formulaInterest}
  \forall z \in \R, \  \E\left( e^{- \mu E((-z,\infty))} \right) = g\left( z - T_{\infty,\mu}\right).
\end{equation}
Letting $\mu\to\infty$ and using again the monotone convergence theorem, we have
\[
  \forall z \in \R, \ \P(\max E \leq -z) = g\left( z - T_{\infty,\infty} \right),
\]
with $T_{\infty,\infty} = -g^{-1}(\P(\max E \leq 0))$. By definition of $g$, we obtain 
\[\P(\max E \leq x) = \E\left(  e^{- c_m S e^{-\alpha x}}\right),\]
where we have set $c_m = e^{-\alpha T_{\infty,\infty}}$. Then as a consequence of Theorem \ref{thm:BIM}, we have
\[
  -\log (1 - g(-z)) =
  \begin{cases}
    \alpha z + o(1) &\text{if  \eqref{eqn:regularCase} holds}\\
    \alpha z - \log(\alpha z) + o(1) & \text{ if \eqref{eqn:boundaryCase} holds}
  \end{cases}
  \quad 
  \text{ as $z \to \infty$.}
\]
Thus $z \mapsto -\log \P(\max E \geq z)$ is uniformly continuous on $[K,\infty)$, for some $K$, using the Heine--Cantor theorem and the above asymptotic.
\end{proof}

Using Lemma \ref{lem:increasingFunctions}, we also prove that the law of the number of points to the right of $z$ conditioned on $\max E \geq z$ converges as $z \to \infty$.
\begin{corollary}
\label{cor:conv}
There exists a probability measure $\nu$ on $\N$ such that
\[
  \forall k \in \N, \quad \lim_{z \to \infty} \P(E((z,\infty)) = k | E((z,\infty)) > 0) = \nu(k).
\]
\end{corollary}

\begin{proof}
We recall from \eqref{eqn:formulaInterest} in the proof of Proposition~\ref{prop:maxDis} that for all $\mu > 0$, we have
\[
  \forall z \in \R, \ \E\left( e^{-\mu E((z,\infty))} \right) = g(-z - T_{\infty,\mu}),
\]
with $T_{\infty,\mu} = - g^{-1}\left( \E\left( e^{-\mu E((0,\infty))} \right)\right)$. As $g$ is decreasing on $\R$, we observe that $T_{\infty,\mu}$ decreases from $\infty$ to $T_{\infty,\infty}$ on $(0,\infty)$.

As a result, we obtain
\[
  \E\left( 1 -  e^{-\mu E((z,\infty))} \middle| E((z,\infty)) > 0 \right) = \frac{1 - g(-z-T_{\infty,\mu})}{1-g(-z-T_{\infty,\infty})}.
\]
Letting $z \to \infty$ and using Theorem~\ref{thm:BIM}, we obtain
\[
  \lim_{z \to \infty} \E\left( 1 -  e^{-\mu E((z,\infty))} \middle| E((z,\infty)) > 0 \right) = e^{-\alpha (T_{\infty,\mu} - T_{\infty,\infty})}.
\]
It yields
\begin{equation}
  \label{eqn:laplaceConv}
  \lim_{z \to \infty} \E\left( e^{-\mu E((z,\infty))} \middle| E((z,\infty)) > 0 \right) = 1 - e^{-\alpha (T_{\infty,\mu} - T_{\infty,\infty})}.
\end{equation}
Therefore, the Laplace transform of the law of $E((z,\infty))$, conditionally on $E((z,\infty)) > 0$, converges pointwise to a continuous function. As a result, we deduce that $E((z,\infty))$ converges in distribution to an integer-valued random variable $X$ with Laplace transform $\E(e^{-\mu X}) = 1 - e^{-\alpha (T_{\infty,\mu} - T_{\infty,\infty})}$.
\end{proof}

We now turn to proving that for any $\phi \in \mathfrak{T}$, the function $F_\phi(\log t)$ is uniformly continuous in a neighbourhood of $\infty$, so Theorem~\ref{thm:ABM} can be used again to characterize the Laplace transform $F_\phi$.

\begin{lemma}
\label{lem:compact}
Let $\phi$ be a compactly supported function in $\mathfrak{T}$. Under assumptions \eqref{eqn:regularityAssumption}, \eqref{eqn:bothCase} and \eqref{eqn:nonLattice} there exists $T_\phi \in \R$ such that 
\[\forall x \in \R, \ F_\phi(x) = \E\left( \exp\left( - S e^{\alpha (x - T_\phi)} \right) \right).\]
\end{lemma}

\begin{proof}
Without loss of generality, we assume that the support of $\phi \in \mathfrak{T}$ is a compact subset of $(0,\infty)$.
Similarly to the previous proof, we set $f(t) = F_\phi(\log t)$ for all $t > 0$. Then with \eqref{eqn:stablePointMeasure}, we have
\[
  f(t) = \E(\prod_{j \geq 1} f(te^{z_j})),
\]
hence $f$ is a fixed point of the smoothing transform. By dominated convergence, $t \mapsto f(t)$ is continuous, hence to apply Theorem~\ref{thm:ABM}, it is enough to prove that 
\[
  h: z \mapsto -\log (1 - f(e^{-z})) = -\log \E\left( 1 - e^{-\crochet{\tau_{-z} E, \phi}} \right)
\]
is uniformly continuous on $[K,\infty)$ for some $K>0$.

Using that $\phi$ has support in $\R_+$, for all $z \in \R_+$, we have
\[
   \E\left( 1 - e^{-\crochet{\tau_{-z} E, \phi}} \right) =  \E\left( 1 - e^{-\crochet{\tau_{-z} E, \phi}} \middle| \max E > z \right)\P(\max E > z).
\]
Therefore, for all $0 \leq z \leq z'$, we have
\begin{equation}
  \label{eqn:towarduc}
  h(z) - h(z') = \log  \frac{\E\left( 1 - e^{-\crochet{\tau_{-z'} E, \phi}} \right)}{\E\left( 1 - e^{-\crochet{\tau_{-z} E, \phi}} \right)}
  = \log  \frac{  \E\left(1 -e^{-\crochet{\tau_{-z'} E, \phi}} \middle| \max E > z \right)}{\E\left(1 - e^{-\crochet{\tau_{-z} E, \phi}} \middle| \max E > z \right)}
\end{equation}
We note that by Corollary~\ref{cor:conv}, we have
\begin{multline}
  \label{eqn:bound}
  1 \geq \limsup_{z \to \infty} \E\left(1 - e^{-\crochet{\tau_{-z} E, \phi}} \middle| \max E > z \right) \\
  \geq \liminf_{z \to \infty} \E\left(1 - e^{-\crochet{\tau_{-z} E, \phi}} \middle| \max E > z \right) \\
  \geq 1 - \lim_{z \to \infty} \E\left( e^{-\Vert \phi \Vert_\infty E((z,\infty))} \middle| \max E > z \right) > \eta
\end{multline}
for some $\eta > 0$.

For all $\delta > 0$, we set $\omega(\delta) = \sup_{y \geq 0, y' \geq 0, |y'-y| < \delta} |\phi(y') - \phi(y)|$. We have $\lim_{\delta \to 0} \omega(\delta) = 0$, by uniform continuity of $\phi$. Then, using the tightness of the law of $E((z,\infty))$ conditionally on $\max E > z$, for all $\epsilon > 0$, there exists $A > 0$ such that
\begin{align*}
  &\left| \E\left( 1 - e^{-\crochet{\tau_{-z'} E, \phi}}\middle| \max E > z \right) - \E\left( 1 - e^{-\crochet{\tau_{-z} E, \phi}} \middle| \max E > z \right) \right|\\ 
  &\leq \E\left( \left| e^{-\crochet{\tau_{-z'} E, \phi}} - e^{-\crochet{\tau_{-z} E, \phi}} \right| \ind{E((z,\infty)) \leq A} \middle| \max E > z \right)+\epsilon\\
  &\leq \E\left( e^{-\crochet{\tau_{-z}E,\phi}} \left| e^{\omega(\delta) E((z,\infty))} - 1 \right| \ind{E((z,\infty)) \leq A} \middle| \max E > z \right) + \epsilon\\ 
  &\leq (e^{\omega(\delta) A} - 1 ) + \epsilon.
\end{align*}
Thus, choosing $\delta > 0$ small enough, for all $z,z' \in \R_+$ with $|z-z'| \leq \delta$, \eqref{eqn:towarduc} yields
\[
  |h(z') - h(z)| \leq \left|\log\left( 1 - \frac{2\epsilon}{\E\left(1 - e^{-\crochet{\tau_{-z} E, \phi}} \middle| \max E > z \right)} \right)\right|.
\]
As a result, in view of \eqref{eqn:bound}, there exists $K>0$ large enough such that
\[
  |h(z') - h(z)| \leq 4\epsilon /\eta
\]
for all $z,z' > K$ with $|z'-z| < \delta$. This proves the uniform continuity of $h$.

Then, applying Theorem \ref{thm:ABM}, we conclude there exists $T_\phi \in \R$ such that $F_\phi(z) = g(z - T_\phi)$ for all $z \in \R$, completing the proof.
\end{proof}

\begin{proof}[Proof of necessity part of Theorem~\ref{thm:main}]
Let $E$ be a random point measure satisfying \eqref{eqn:stablePointMeasure}. By Lemma~\ref{lem:compact}, for any $\phi \in \mathfrak{T}$ with compact support, there exists $T_\phi \in \R$ such that
\[
  \E\left( \exp\left( - \crochet{\tau_z E,\phi} \right) \right) = \E\left( e^{-Se^{\alpha (z - T_\phi)}} \right) = g(z - T_\phi).
\]
Additionally, by Theorem~\ref{thm:BIM}, the function $g$ satisfies \eqref{eqn:21}. Therefore, by Theorem~\ref{thm:SuZ}, there exists $c>0$ such that $E$ is a SDPPP($cS$, $e^{-\alpha x}\dd x$, $\mathcal{D}^\ast$), with $\mathcal{D}^\ast$  the distribution defined in \eqref{eqn:obtainingTheDecoration}.
\end{proof}

\bibliographystyle{alpha}

\end{document}